\documentclass{amsart}
\usepackage[utf8]{inputenc}
\usepackage{amsmath}
\usepackage{amssymb}
\usepackage{amsthm}
\usepackage{amsopn}
\usepackage{color}
\usepackage{nicefrac}
\usepackage{enumerate}
\usepackage{enumitem}
\usepackage[hidelinks]{hyperref}
\usepackage{soul}

\usepackage{setspace}
\newtheorem{theorem}{Theorem}[section]

\newtheorem{lemma}[theorem]{Lemma}

\newtheorem{proposition}[theorem]{Proposition}
\newtheorem{corollary}[theorem]{Corollary}
\theoremstyle{remark}
\newtheorem{remark}[theorem]{Remark}

\newcommand{\R}{\mathbb{R}}
\newcommand{\RN}{\mathbb{R}^N}
\newcommand{\locLip}{W^{1,\infty}_{\textit{loc}}}

\title[]{Large-time behavior of unbounded solutions of the viscous Hamilton-Jacobi equation: quadratic and  subquadratic cases}
\author{Alexander Quaas and Andrei Rodr\'{\i}guez-Paredes}

\begin{document}
	\maketitle

\begin{abstract}
We determine the large-time behavior of unbounded solutions for the so-called \textit{viscous Hamilton Jacobi equation}, $u_t - \Delta u + |Du|^m = f(x)$, in the \textit{quadratic} and \textit{subquadratic} cases (i.e., for $1<m\leq 2$), with a particular focus on allowing arbitrary growth at infinity for $f$ and the prescribed initial data. The lack of a comparison principle for the associated \textit{ergodic problem} is overcome by proving that a \textit{generalized simplicity} holds for sub- and supersolutions of the ergodic problem. Moreover, as the uniqueness of solutions of the parabolic problem remains open in the current setting, our result on large-time holds for \textit{any} solution, even if multiple solutions exist.
\end{abstract}

\vspace*{1em}

	
\section{Introduction}
    The following work deals with the parabolic \textit{viscous Hamilton-Jacobi Equation},
        \begin{align}
            u_t - \Delta u + |Du|^m = f(x) &{}\quad\textrm{in } \RN\times (0,+\infty), \label{vhj_whole}\\
            u(x, 0) = u_0(x) &{}\quad\textrm{in } \RN,\label{initialData_whole}
        \end{align}
    in the the \textit{subquadratic} and \textit{quadratic} cases, i.e.~when $1<m\leq 2$. The functions $f, u_0$ are continuous and assumed bounded from below, but not from above. Thus, in this setting, solutions to \eqref{vhj_whole}-\eqref{initialData_whole} are in general unbounded. Further assumptions on $f, u_0$ are stated as needed.
        
    Our main interest is determining the large-time behavior of solutions to \eqref{vhj_whole}-\eqref{initialData_whole} for data---$u_0$ and $f$---as general as possible, particularly with the most general growth conditions at infinity.
    
    We recall that the problem of large-time behavior for \eqref{vhj_whole} set on a \textit{bounded} domain (along with the necessary boundary data) is addressed in Tabet \cite{tabet2010large} for the superquadratic case, then by this same author in collaboration with Barles and Porretta in \cite{barles2010large} for the subquadratic case. While some key ideas from these works remain useful in the present setting, there is considerable difficulty in obtaining similar results for unbounded solutions over an unbounded domain. 
    
    There are several works that address the large-time behavior of solutions of \eqref{vhj_whole} set in the whole space in the context of weak solutions, but by necessity these work within the class of bounded solutions; we refer the reader to \cite{barles2001space, benachour2004asymptotic, biler2004asymptotic, gallay2007asymptotic, iagar2010asymptotic, laurenccot2009non} and the works cited therein, though this list is by no means exhaustive.\\
    
    For unbounded solutions over the whole space, Ichihara \cite{ichihara2012large} obtains results on large-time behavior for a class of equations which includes the model \eqref{vhj_whole}, for all $m>1$, by combining probabilistic and PDE techniques. Additionally, these results make use of precise growth assumptions for both $f$ and $u_0$: roughly, it is assumed in \cite{ichihara2012large} that $f(x)\approx |x|^{m^*}$, where $m^*=\frac{m}{m-1}$ (the conjugate exponent of $m$) and that $u_0$ has at most polynomial growth.
    
    The results of \cite{ichihara2012large} show that the large-time behavior of solutions of \eqref{vhj_whole} is determined by the associated \textit{ergodic problem},
   		\begin{equation}\label{erg_prob}\tag{\textit{EP}}
   			\lambda -\Delta \phi + |D\phi|^m = f(x)\quad \textrm{in }\RN,
   		\end{equation}      
    also studied therein, where both $\lambda$ and $\phi$ are unknown. More precisely, it is shown in \cite{ichihara2012large} that the solution of \eqref{vhj_whole} behaves like $\phi(x) + \lambda t$ as $t\to+\infty$, where $(\lambda, \phi)$ is the unique solution pair of \eqref{erg_prob} within the class of bounded-from-below functions with at most polynomial growth (here, $\phi$ is unique up to an additive constant). This result if commonly referred to as \textit{ergodic large-time behavior}, and is in accordance with the bounded domain case.
    
    Given the importance of the ergodic problem, we comment with some detail on some of the previous results in the unbounded setting. Also, we will sometimes write ``\eqref{erg_prob}$_\lambda$'' to stress the dependence of the problem on the constant $\lambda$; in particular, when a given value for $\lambda$ is considered.
    
    The results of \cite{ichihara2012large} were extended by Barles and Meireles \cite{barles2016unbounded} in terms of the generality allowed for the data, though only the model equation \eqref{vhj_whole} is addressed. In particular, in the subquadratic case, it is shown that uniqueness of solutions for \eqref{erg_prob} holds if $f(x)\approx |x|^\alpha$---where $\alpha>0$ is now arbitrary---and $Df$ behaves accordingly (see \cite{barles2016unbounded}, Theorem 3.10 and Proposition 3.7).
    
    In contrast, in a recent work by Arapostathis, Biswas and Caffarelli \cite{arapostathis2019uniqueness}, uniqueness of solutions for \eqref{erg_prob} is obtained assuming only a technical hypothesis which relates $f$ and $Df$, but still allows $f$ to have arbitrary growth---see \ref{assmp:growth_Df} below. This is accomplished by means of the uniqueness of a measure related to the underlying dynamics of the problem. However, no comparison principle is proved for \eqref{erg_prob}, and in fact no information is provided on the behavior of \textit{sub-} and \textit{supersolutions}. These are crucial elements towards obtaining large-time behavior for \eqref{vhj_whole} in our approach and in many of the cited works.
    
    On this matter, we have found that a \textit{generalized simplicity} holds for \textit{sub-} and \textit{supersolutions} of \eqref{erg_prob}. We use ``simplicity'' here in the spirit of similar results for e.g.~the principal eigenfunction of second-order elliptic operators (see e.g.~\cite{berestycki1994principal}); and ``generalized'' since the property holds for also for sub- and supersolutions, and not merely for solutions. These results are proved in Propositions \ref{prop_simplicity} and \ref{prop_simplicity_below} and, as expected, are an essential part of the proof of our main result, Theorem \ref{thm_main}.\\
     
    Returning to the question of large-time behavior, the authors studied this problem in the superquadratic case ($m>2$) in \cite{barles2020large}, in collaboration with G.~Barles. In \cite{barles2020large} it is shown that ergodic large-time behavior occurs assuming minimal hypotheses on $f$ and $u_0$, although in this case we are able to prove a very general comparison principle for the parabolic problem, and a similar result was already available for the ergodic problem. Regarding the data, for $f$ it is assumed that there exists a  nondecreasing function $\varphi: [0, +\infty) \to [0, +\infty)$ and constants $\alpha, c>0$ such that for all $r\geq 0$, $c^{-1} r^\alpha \leq \varphi(r)$, and for all $x\in \RN$ and $r=|x|$, $c^{-1} \varphi(r) -c \leq f(x) \leq c (\varphi(r) + 1)$. In other words, $f$ is controlled by some radial function that has \textit{at least} some polynomial growth, but is not bounded in any way by above. 
    
    In the present setting, we are able to obtain the result on large-time behavior without the assumption of a ``minimum growth rate'' for $f$. We discuss this further after providing the precise statements of our results.
   
    \subsection{Assumptions and main results}
    
    Our main result establishes, for $1<m\leq 2$, that the large-time behavior of solutions of \eqref{vhj_whole}-\eqref{initialData_whole} is determined by the solutions of \eqref{erg_prob}. To this end, we require the following assumptions on $f$.
	 
	 \begin{enumerate}[label=(H\arabic*)]        
        \item\label{assmp:growth_f} $f\in \locLip(\RN)$ is coercive; i.e., $f(x)\to +\infty$ as $|x|\to+\infty$.
	 	\item\label{assmp:growth_Df} $\limsup\limits_{x\to +\infty} \frac{|Df(x)|^{\frac{1}{2m-1}}}{|f(x)|^{\frac{1}{m}}} < +\infty.$ 
	 \end{enumerate}
     
     \begin{remark}
        Assumption \ref{assmp:growth_Df} may be interpreted as preventing $f$ from having large oscillation at infinity. The fact that $\frac{1}{2m-1}<\frac{1}{m}$ for all $m>1$ allows $f$ to have arbitrary growth. 
     \end{remark}
     
     We also introduce the \textit{generalized ergodic constant},
        \begin{equation}\label{def_lambda^*}
            \lambda^* = \sup \{\lambda\in \R \ | \ \text{\eqref{erg_prob}$_\lambda$ has a }C^2\textrm{-subsolution}\}.
        \end{equation}
    The relevant properties of $\lambda^*$ are provided in Section \ref{sec_prelim}.
    
    The following is a version of the simplicity results for \eqref{erg_prob}; for more results of this type, see Section \ref{sec_prelim} below.
            
        \begin{proposition}\label{prop_simplicity}
            If  $\chi$ and $v\in USC(\RN)$ are respectively a solution and a subsolution of \eqref{erg_prob}$_{\lambda^*}$, both bounded from below, then there exists $c\in \R$ such that $v(x) = \chi(x) + c$ for all $x\in \RN$.
        \end{proposition}
     
    The proof of Proposition \ref{prop_simplicity} follows from combining an approximation of \eqref{erg_prob} over bounded domains with an appropriately chosen perturbation of solutions; this perturbation relies on the homogeneity of the gradient term. In the end, these computations allow us to exploit the coercivity of $f$ and ``order solutions at infinity''. It is remarkable that this result can be obtained from such elementary techniques.
    
    \medskip

      \begin{theorem}\label{thm_main}
         Assume $u_0\in C(\RN)$ is bounded from below and assumptions \ref{assmp:growth_f} and \ref{assmp:growth_Df} hold. Then, there exists a constant $\hat{c}\in \R$ such that 
             \begin{equation*}
                 u(\cdot,t) - \lambda^*t \to \phi \text{ locally uniformly over } \RN,
             \end{equation*}
         where $u=u(x,t)$ is any solution of \eqref{vhj_whole}-\eqref{initialData_whole}, $\lambda^*$ is defined by \eqref{def_lambda^*} and $\phi$ is the unique solution of \eqref{erg_prob}$_{\lambda^*}$ given by Proposition  \ref{prop_simplicity}.
      \end{theorem}
     
         \begin{remark} Recall that in the case $1<m\leq 2$ there is no comparison principle available for neither the ergodic nor the parabolic problems. In particular, the uniqueness of solutions for the parabolic problem---in the generality of the current setting---remains an open problem. Nevertheless, our result on large time behavior holds for any solution of \eqref{vhj_whole}.
          \end{remark}

    The proof of our main result, Theorem \ref{thm_main}, begins by showing that a half-relaxed limit for $u-\lambda^* t$ converges to the solution $\phi$ of \eqref{erg_prob} over a subsequence $t_k\to +\infty$; it is here that Proposition \ref{prop_simplicity} is essential. From here we construct parabolic barriers for some new initial condition related to the previous subsequence, again using some of the ideas behind the proof of the simplicity results.\\
    
    \noindent\textbf{Notation and organization of the article.} The notation used in the article is entirely standard. Section \ref{sec_prelim} contains preliminary results concerning \eqref{vhj_whole} and \eqref{erg_prob}, including the proof of Proposition \ref{prop_simplicity} as well as others which are collected from previous works for convenience. Our main result is proved in Section \ref{sec_large}.

   \section{Preliminaries}\label{sec_prelim}
   
   We start by establishing some existence results for the parabolic problem. 
   
    \begin{theorem}[Existence of solutions]\label{thm_exist_parab}
        Assume $f, u_0\in \locLip(\RN)$ are bounded from below. Then, there exists a continuous, bounded from below solution of \eqref{vhj_whole}-\eqref{initialData_whole}.
    \end{theorem}
    
    \begin{proof}
        The proof is essentially the same as that of Proposition 3 in \cite{barles2020large} for the case $m>2$. We repeat the main points of the argument for convenience. Also, we note that the difference in the generality allowed for the initial $u_0$ in the is due to the lack of a comparison result for \eqref{vhj_whole} in the present case, $m\leq 2$. See Remark \ref{locLip_data} below.
        
        Let $\{T_R\}$ be a sequence such that where $T_R>0$ for all $R>0$ and $T_R\nearrow +\infty$ as $R\to +\infty$, and consider the \textit{parabolic state-constraints problem} on $B_R \times (0,T_R)$, 
    		\begin{align}
    			u_t - \Delta u + |Du|^m = f(x) &{} \quad\textrm{ in } B_R \times (0,T_R), \label{eqonball}\\
    			u_t - \Delta u + |Du|^m \geq f(x) &{} \quad\textrm{ in } \partial B_R \times (0,T_R), \label{sconball}\\
    			u(\cdot,0) = u_0 &{} \quad\textrm{ in } \overline{B}_R. \label{initialonball}
    		\end{align}
        
        By the results of \cite{barles2004generalized}, this problem has a unique continuous viscosity solution, which we denote by $u^R$.
        
        It can be shown that if $R'\geq R$, then $u^{R'}(x, t) \leq u^R(x,t)$ for all $(x,t)\in B_R\times [0,T_R]$. This follows from the fact that solutions of \eqref{eqonball} satisfying \eqref{sconball} are maximal. Additionally, a constructive proof of this is provided in Lemma 2 in \cite{barles2020large}, and it is equally valid for $1<m\leq 2$. The solution of \eqref{vhj_whole}-\eqref{initialData_whole} can be then obtained as a locally uniform limit of the solutions $u^R$ as $R\to +\infty$ as follows.
        
        Consider a fixed $\bar R$. For $R>2\bar R +1$, we have, assuming $u_0,f\geq 0$
            \begin{equation*}
                0 \leq u^R \leq u^{2\bar R +1} \quad \hbox{in  }\overline{B_{2\bar R}} \times [0,T_{2\bar R}]\;
            \end{equation*}
        hence $\{u^R\}_{R> 2\bar R +1}$ is uniformly bounded over $\overline{B_{2\bar R}} \times [0,T_{2\bar R}]$.
         
        Furthermore, using Theorem~\ref{gradBound} and Corollary \ref{gradBoundCor} in the Appendix, the $C^{0,1/2}$-norm
        of $u^R$ on $\overline{B_{\bar R}} \times [0,T_{2\bar R}]$ remains also uniformly bounded for $R>2\bar R +1$.
        
        Thus, by using the Ascoli-Arzela Theorem, we have the uniform convergence of $u^R$ 
        on $\overline{B_{\bar R}} \times [0,T_{2\bar R}]$. Then, by a diagonal argument, we may extract a subsequence of $\{u^R\}_{R>0}$ that converges locally uniformly to some $u\in C(\RN\times [0,+\infty)$. By stability, it follows that $u$ is a viscosity solution of \eqref{vhj_whole}-\eqref{initialData_whole}.
    \end{proof}
    
    \begin{remark}\label{locLip_data}
       \textit{(i)}   In \cite{barles2020large}, existence of solutions for all bounded-from-below $u_0\in C(\RN)$ was obtained by approximating $u_0$ with functions in $\locLip(\RN)$ and employing a result analogous to Theorem \ref{thm_exist_parab}. A comparison result that is not available is a crucial tool in this procedure, since it is used to control the ``error'' between the corresponding solutions in terms of the error in the approximation of the initial data.
       \textit{(ii)} Theorem \ref{thm_main} holds for any solution of the parabolic problem that may also exist for $u_0\in C(\RN)$. 
    \end{remark}

    \begin{theorem}\label{thm_lambda^*}
        Assume that $f\in \locLip(\RN)$ is bounded from below. Then $\lambda^*$ as defined in \eqref{def_lambda^*} is finite and \eqref{erg_prob}$_{\lambda^*}$ has a $C^2$-solution.
    \end{theorem}
    
    \begin{proof}
        See Theorem 2.4 in \cite{barles2016unbounded}.
    \end{proof}        
    
    \begin{theorem}\label{thm_bar_lambda}
        Assume that $f\in \locLip(\RN)$ is coercive, i.e.,
            \begin{equation*}
                f(x)\to +\infty \quad\text{as}\quad|x|\to+\infty.
            \end{equation*}
        Then \eqref{erg_prob} has a $C^2$-solution which is bounded from below.
    \end{theorem}
        
    \begin{proof}
        See Theorem 2.6 in \cite{barles2016unbounded}. For further reference, we note that the desired solution is therein obtained as the limit of solutions to the ergodic problem set in $B_R$. More specifically, if $(\lambda_R, \phi_R)$ satisfy
           \begin{equation}\label{ergodic_SC_BR}
               \left\{ 
                   \begin{array}{cl}
                       \lambda_R -\Delta \phi_R + |D\phi_R|^m = f(x) &  \textrm{in }B_R\\
                       \lambda_R -\Delta \phi_R + |D\phi_R|^m \geq f(x) &  \textrm{on }\partial B_R.
                   \end{array}
               \right.
           \end{equation}
        then for some $\bar\lambda\in \R$ and $\bar\phi\in C^2$, $\lambda_R\searrow \bar\lambda$ and $\phi_R - \phi_R(0)\to \bar\phi$ locally uniformly in $C^2$. By stability, it follows that $(\bar\lambda, \bar\phi)$ is a $C^2$-solution of \eqref{erg_prob}. Furthermore, we stress that the sequence $(\lambda_R)$ is nonincreasing with respect to $R$ and that for all $R>0$, $\lambda_R\geq \min_{\RN} f$ (thus $\bar\lambda\geq \min_{\RN} f$ as well). Both facts follow from the extremal characterization of $\lambda_R$, analogous to \eqref{def_lambda^*}; namely
            \begin{equation}\label{char_lambda_R}
                \lambda_R = \sup \{\lambda  \ | \ \exists u: \ \lambda u - \Delta u + |Du|^m \leq f(x) \text{ in } B_R \}.
            \end{equation}
    \end{proof}
    
    \begin{proposition}\label{prop_low_conv_R}
        Denote by $(\psi_R, \nu_R)$ the solution pair of
            \begin{equation*}
                \lambda + \Delta \psi + |D\psi|^m = f_R \quad\textrm{in } \RN/ 2S_R\mathbb{Z}^N,
            \end{equation*}
        where $S_R$ is chosen so that $f(x) \geq R$ for $|x|\geq S_R$ and $f_R$ is the periodic extension of $\min\{f, R\}$ to $\RN/2S_R\mathbb{Z}^N$.
        
        Then $\nu_R\to \bar\nu$ and $\psi_R - \psi_R(0)\to \bar{\psi}$ locally uniformly in $C^2$, where $(\bar\psi, \bar\nu)$ are a solution of \eqref{erg_prob}, and $\bar{\psi}$ is bounded from below.
    \end{proposition}
    
    \begin{proof}
        See \cite{barles2016unbounded}, Proposition 4.2.
    \end{proof}
        
    \medskip
    
    \begin{lemma}\label{lemma:supersol_growth} 
        Assume \emph{\ref{assmp:growth_f}}.
            \begin{enumerate}[label=\textit{(\roman*)}]
                \item If $\chi$ is a bounded-from-below supersolution of \eqref{erg_prob}, then
                 			\begin{equation*}
                 				\lim\limits_{|x|\to +\infty} \frac{\chi(x)}{|x|} = +\infty.
                 			\end{equation*}
               \item Let $t_0>0$. If $v$ is a bounded-from-below supersolution of \eqref{vhj_whole}, then 
                \begin{equation*}
                    \lim\limits_{|x|\to +\infty} \frac{v(x, t_0)}{|x|} = +\infty.
                \end{equation*}
            \end{enumerate}
       	\end{lemma}
     	
     	\begin{proof}  We will prove only \textit{(ii)}, since part \textit{(i)} is analogous. Define, for all $r\geq 0$, $\varphi(r) = \inf_{\RN\backslash B_r} f$. Observe that $\varphi$ is non-decreasing and, for all $x\in \RN$, $f(x)\geq \varphi(|x|)$. Furthermore, assumption \ref{assmp:growth_f} implies $\varphi$ is coercive. We will argue by contradiction. 
         
         Assume that there exists some $t_0>0$ and a sequence $(y_n)_{n\in \mathbb{N}}$ with $|y_n| \to +\infty$ as $n\to +\infty$ such that $\nicefrac{v(y_n, t_0)}{|y_n|}$ remains bounded as $n\to+\infty$. Since $\varphi$ is coercive, this implies 
    	 		\begin{equation}\label{v_low_contrad}
    		 		\liminf\limits_{n\to +\infty} \frac{v(y_n, t_0)}{|y_n|\varphi(y_n)^{\frac{1}{m}}} = 0.
    	 		\end{equation}
     		Define for $y \in B_1$,
    	 		\begin{equation*}
    		 		v_n(y, s) = \frac{v(y_n + \frac{|y_n|}{2} y, t_0 + \frac{|y_n|^2}{4}s)}{|y_n| \varphi(\frac{1}{2}|y_n|)^{\frac{1}{m}}}.
    	 		\end{equation*}
    	 		
     		Formally, $v_n$ satisfies
    	 		\begin{align*}
    		 		&{}2^{2-m} |y_n|^{-1} \varphi(\frac{1}{2}|y_n|)^{\frac{1}{m} - 1} \left(\partial_t v_n -\Delta v_n\right) \, + \, |Dv_n|^m\\
    		 		&{}\quad = 2^{-m}\varphi(\frac{1}{2}|y_n|)^{-1}f\left(y_n + \frac{|y_n|}{2}y\right)\geq  2^{-m}
    	 		\end{align*}
     		where the last inequality follows from the definition of $\varphi$ and the fact that $|y_n + \frac{|y_n|}{2}y| \geq |y_n| - \frac{|y_n|}{2}|y| = \frac{1}{2}|y_n|$; it is not hard to pass the computation onto a suitable test function. Writing $\epsilon_n= 2^{2-m} |y_n|^{-1} \varphi(\frac{1}{2}|y_n|)^{\frac{1}{m} - 1}$, we have $\epsilon_n\to 0$ since $\varphi$ is coercive and $1 - \frac{1}{m} < 0$, hence
    	 		\begin{equation}\label{vn_eqn}
    		 		\epsilon_n ( \partial_t v_n - \Delta v_n) \, + \, |Dv_n|^m \geq 2^{-m} -o_n(1) \quad \hbox{in  }B_1 \times (t_0,+\infty),
    	 		\end{equation}
     		where $o_n(1)\to 0$ as $n\to \infty$.
     		
     		In order to pass to the limit, we lack some $L^\infty$-bound on $v_n$. To overcome this difficulty, we set
    	 		\begin{equation*}
    		 		\tilde{v}_n (y,s) = \min\left (v_n (y,s), K(1-|y|)\right) \quad \hbox{in  }B_1 \times (0,+\infty),
    	 		\end{equation*}
     		for some suitably large $K$. For $n$ large enough, the concave function $y \mapsto K(1-|y|)$ is also a supersolution of \eqref{vn_eqn} and therefore so is $\tilde  v_n$, as the minimum of two supersolutions.
     		
     		Hence, the half-relaxed limit $\tilde v = \liminf\limits_{n\to\infty}\,\!\!^* \,\tilde v_n$ is well defined and, by stability, we have in the limit $n\to +\infty$ that
    	 		\begin{equation*}
    		 		|D\tilde v|^m \geq 2^{-m} \quad \textrm{in } B_1\times (0,+\infty),
    	 		\end{equation*}
     		in the viscosity sense. For all $n$, $\tilde v_n \geq 0$, hence also $\tilde v\geq 0$ (recall that $v\geq 0$ may be assumed at the outset). Thus $\tilde v$ is a supersolution of the eikonal equation, $|Du|^m = 2^{-m}$ in $B_1$ with homogeneous boundary condition. The latter has the unique solution $\frac{1}{2}d(y, \partial B_1)$. Therefore, by comparison, $\tilde v(y,s) \geq \frac{1}{2}d(y, \partial B_1)$ for all $y\in B_1, \ s\geq 0$. We remark that comparison holds up to $s=0$ by the (more general) results of \cite{da2004remarks}.
             
            Finally, using \eqref{v_low_contrad} we have
    	 		\begin{equation*}
    		 		0 \leq \tilde v(0,0) \leq  \liminf\limits_{n\to +\infty} \frac{v(y_n,t_0)}{|y_n| \varphi(\frac{1}{2}|y_n|)^{\frac{1}{m}}} = 0.
    	 		\end{equation*}
     		This implies $0 = \tilde v(0,0) \geq \frac{1}{2}d(0, \partial B_1) = \frac{1}{2}$, a contradiction.
     	\end{proof}

        \begin{lemma}[``Simplicity'']\label{lemma_simplicity}
            Let $(\bar\lambda, \bar\phi)$ denote the solution pair of \eqref{erg_prob} obtained in Theorem \ref{thm_bar_lambda} and $v\in USC(\RN)$ be a bounded-from-below subsolution of \eqref{erg_prob}$_{\bar\lambda}$. Then, there exists $c\in \R$ such that $v(x) = \phi(x) + c$ for all $x\in \RN$.
        \end{lemma}
        
        \begin{proof}
            We will first establish the following
                \begin{equation}\label{simplicity_ineq}
                    \textrm{there exists } c\in \R \textrm{ such that } {v}(x) \leq \phi(x) + c\quad \textrm{for all }x\in \RN.
                \end{equation}
            To this end, we first prove that $\min_{\bar B_R} \phi_R -{v}$ is bounded from below, uniformly in $R>0$. Let $\mu >1$---to be chosen later---and define $\phi_R^\mu = \mu\phi_R$, where $(\lambda_R, \phi_R)$ are defined as in the proof of Theorem \ref{thm_bar_lambda}. So defined, $\phi_R^\mu$ is a solution of
                \begin{equation*}
                    -\Delta \phi_R^\mu + \mu^{1-m}|D\phi_R^\mu|^m = \mu(f(x) - \lambda_R)\quad\textrm{in } B_R.
                \end{equation*}
            Since $\mu>1$ and $1-m<0$, we have $\mu^{1-m}\leq 1$, thus
                \begin{equation}\label{eq_mu_R_v}
                    -\Delta \phi_R^\mu + |D\phi_R^\mu|^m \geq \mu(f(x) - \lambda_R)\quad\textrm{in } B_R.
                \end{equation}
            
            Set $\mu = \mu_R := 1 + \lambda_R - \bar\lambda$. By the results of \cite{barles2016unbounded} (see also the proof of Theorem \ref{thm_bar_lambda} above), we indeed have that $\mu_R > 1$, as well as $\mu_R\searrow 1$ as $R\to+\infty$. 
            
            Consider $x^R \in \mathrm{argmax}_{\overline{B}_R} (v- \phi_R^{\mu_R} )$. This maximum is achieved in $B_R$ because $v\in USC$ and $\phi_R^{\mu_R}\to +\infty$ as $x\to \partial B_R$. Since $v$ is a subsolution of \eqref{erg_prob}$_{\bar\lambda}$ and $\phi_R^{\mu_R}\in C^2$, we have
                \begin{equation*}
                    \bar\lambda - \Delta \phi_R^{\mu_R}(x_R) + |D\phi_R^{\mu_R}(x_R)|^m \leq f(x_R).
                \end{equation*}
            Evaluating \eqref{eq_mu_R} at $x_R$ and subtracting from the previous equation we obtain $0 \geq (\mu_R - 1)f(x^R) - \mu_R\lambda_R + \bar\lambda$. From the definition of $\mu_R$, we have
                \begin{align*}
                    &(\mu_R - 1)f(x^R) \leq  \mu_R\lambda_R - \bar\lambda = (1 + \lambda_R - \bar\lambda) \lambda_R - \bar\lambda\\
                    &\quad = (\lambda_R -\bar{\lambda})(1 + \lambda_R),
                \end{align*}
            hence
                \begin{align}\label{armin_ineq_1}
                    f(x_R) \leq 1+\lambda_R \leq 1 + \lambda_1.
                \end{align}
            Here we have used the monotonicity of $\lambda_R$ (to clarify, $\lambda_1$ denotes the ergodic constant for the domain $B_1$). Since the right-hand side is bounded uniformly with respect to $R$ and $f$ is coercive, $x_R$ must remain in a bounded set as $R\to+\infty$, say $B_{\bar{R}}$ for some $\bar{R}>0$. We stress that $\bar{R}$ is independent of $R$ (thus also of $\mu_R$).
            
            Therefore, for $R>\bar{R}$,
                \begin{align*}
                    &\max_{B_R} (v - \phi_R^{\mu_R}) = (v - \phi_R^{\mu_R})(x_R)  = \max_{B_{\bar{R}}} (v-\phi_R^{\mu_R}) \leq \left(\max_{B_{\bar{R}}} (v-\bar\phi) \right) - 1  =: c,
                \end{align*}
            where we have used that $\phi_R^{\mu_R}\to \bar\phi$ uniformly over $B_{\bar{R}}$ as $R\to+\infty$ (recall $\mu_R\to 1$ in said limit). This implies that
                \begin{equation*}
                     v(x) \leq \phi_R^{\mu_R}(x) + c \quad\textrm{for all } x\in \RN.
                \end{equation*}
            For any fixed $x \in \RN$, sending $R\to +\infty$ leads to \eqref{simplicity_ineq}, noting $c$ is independent of $R$.\\
            
            We now define $\bar{\phi}^\mu := \mu \bar{\phi}$, where again $\mu>1$; thus,
                \begin{equation}\label{eq_phi_mu}
                    -\Delta\bar{\phi}^\mu + |D\bar{\phi}^\mu| \geq \mu (f(x) - \bar\lambda) \quad \textrm{in }\RN.
                \end{equation}
            Using \eqref{simplicity_ineq}, we have that
                \begin{equation*}
                    v(x) - \bar\phi^{\mu}(x) = v(x) - \bar\phi(x) + (1-\mu)\bar\phi(x) \leq c + (1-\mu)\bar\phi(x) \quad \textrm{for all } x\in \RN.
                \end{equation*}
            Since $\bar\phi$ is coercive (by Lemma~\ref{lemma:supersol_growth}) and $1-\mu<0$, it follows that $v(x) - \bar\phi^\mu(x) \to -\infty$ as $|x|\to+\infty$. Therefore, $v-\bar\phi^\mu$ attains a global maximum at some $x^\mu$, with $x^\mu$ possibly depending on $\mu$. However, proceeding as before, we obtain that 
                \begin{equation*}
                    \bar\lambda - \Delta\bar{\phi}^\mu(x^\mu) + |D\bar{\phi}^\mu(x^\mu)|^m \leq f(x^\mu).
                 \end{equation*}
           Subtracting from \eqref{eq_phi_mu} and using the coercivity of $f$, we again have that $x^\mu$ remains in some ball $B_{\tilde{R}}$, with $\tilde{R}$ depending only on $f$. Hence, for all $\mu>1$, $x^\mu \in \overline{B}_R$ and (up to a subsequence) $x^\mu\to \hat{x}$ for some $\hat{x}\in \overline{B}_{\tilde{R}}$.
                
            Recalling now that each $x^\mu$ is a global maximum, we have, for all $\mu>1$,
                \begin{equation*}
                    \bar{v}(x) - \bar\phi^\mu(x) \leq v(x^\mu) - \bar\phi^\mu (x^\mu) \quad\textrm{for all } x\in \RN.
                \end{equation*}
            Fixing $x\in\RN$ and taking the limit $\mu\to 1^+$ in the previous inequality, we obtain
                \begin{equation*}
                    \bar{v}(x) - \bar\phi(x) \leq \bar{v}(\hat{x}) - \phi(\hat{x}) \quad\textrm{for all } x\in \RN,
                \end{equation*}
            recalling that $v\in USC$. We thus have that $\bar v$ and $\bar\phi + \bar{v}(\hat{x}) - \phi(\hat{x})$ are respectively sub- and supersolution of \eqref{erg_prob}$_{\bar\lambda}$, satisfy $\bar v \leq \bar\phi + \bar{v}(\hat{x}) - \phi(\hat{x})$ in $\RN$, and are equal at $\hat{x}$. Therefore, they are equal by the strong comparison principle (see e.g.~\cite{giga2005strong}, Theorem 3.1 and Remark 3.6), and with this we conclude.
        \end{proof}
        
        \begin{corollary}\label{cor_lambdas}
            Using the notation of Theorem \ref{thm_lambda^*} and the proof of Theorem \ref{thm_bar_lambda}, $\bar\lambda = \lambda^*$.
        \end{corollary}
        
        \begin{proof}
            The fact that $(\bar\lambda, \bar\phi)$ is a solution pair of \eqref{erg_prob} and the definition of $\lambda^*$ given by \eqref{def_lambda^*} imply at once that $\bar\lambda \leq \lambda^*$. 
            
            Assume $\bar\lambda < \lambda^*$. Again by \eqref{def_lambda^*}, there exist $\tilde\lambda\in (\bar\lambda, \lambda^*)$  and $\tilde v\in C^2(\RN)$ such that 
                \begin{equation*}
                    \tilde\lambda - \Delta \tilde v + |D\tilde v|^m \leq f(x) \quad \text{in }\RN.
                \end{equation*}
            But $\bar\lambda <\tilde \lambda$ implies that
                \begin{equation*}
                    \bar\lambda - \Delta \tilde v + |D\tilde v|^m < f(x) \quad \text{in }\RN,
                \end{equation*}
            therefore, by Lemma \ref{lemma_simplicity}, $\tilde v \equiv \bar \phi + c$ for some $c\in \R$. We thus obtain
                \begin{equation*}
                    \bar\lambda - \Delta \tilde v + |D\tilde v|^m = f(x) \quad \text{in }\RN,
                \end{equation*}
            a contradiction.
        \end{proof}
        
        \begin{proof}[Proof of Proposition \ref{prop_simplicity}]
            The result is an immediate consequence of Lemma \ref{lemma_simplicity} and Corollary \ref{cor_lambdas}. It shows, however, the full reach of the previous arguments.
        \end{proof}
    
   \medskip 
    \begin{theorem}[Uniqueness of the ergodic constant]\label{thm_uniqueness}
        Assuming \ref{assmp:growth_f} and \ref{assmp:growth_Df}, the \textit{generalized ergodic constant} of \eqref{erg_prob} is unique in the following sense: if $(\lambda, \psi)$ is a solution of \eqref{erg_prob} and $\psi$ is bounded from below, then $\lambda = \lambda^*$.
    \end{theorem}
    
    \begin{proof}
        For $m=2$, this result is contained in \cite{barles2016unbounded}, Theorem 3.1; for $m\in (1,2)$,  it is covered in part (\textit{a}) of Theorem 2.1 in \cite{arapostathis2019uniqueness}.
    \end{proof}
    
    \begin{corollary}[Uniqueness of solutions for \eqref{erg_prob}]\label{cor_uniqueness}
        Assuming \ref{assmp:growth_f} and \ref{assmp:growth_Df}, solutions for \eqref{erg_prob} are unique up to an additive constant. More specifically, if $(\lambda, \phi)$ and $(\psi, \nu)$ are two bounded-from-below solutions of \eqref{erg_prob}, then $\lambda=\nu$ and $\psi(x) = \phi (x) + c$ for all $x\in \RN$ for some $c\in \R$.
    \end{corollary}
    
    \begin{proof}
        This an easy consequence of Theorem \ref{thm_uniqueness} and the arguments of Lemma \ref{lemma_simplicity}.
    \end{proof}
      
    \begin{proposition}\label{prop_simplicity_below}
            Assume \ref{assmp:growth_f} and \ref{assmp:growth_Df}. Let $(\lambda^*, \phi)$ denote the unique solution pair of \eqref{erg_prob} given by Theorem \ref{thm_uniqueness} and $w\in LSC(\RN)$ be a supersolution of \eqref{erg_prob}$_{\lambda^*}$. Then, there exists $c\in \R$ such that $w(x) = \phi(x) + c$ for all $x\in \RN$.
    \end{proposition}
    
    \begin{remark}
        The statement of Proposition \ref{prop_simplicity_below} is of course symmetrical to that of Lemma \ref{lemma_simplicity}. We have chosen to state them separately since we are not able to obtain the analogue of Corollary \ref{cor_lambdas} for supersolutions. See Remark \ref{rmk_differences_of _approximations} below.
    \end{remark}
    
        \begin{proof}
            The proof is similar to that of Lemma \ref{lemma_simplicity}, so we merely point out a couple of differences in the construction.
            
            For $\gamma<1$, let $\psi_R^\gamma:= \gamma \psi_R$, where $\psi_R$ is defined as in Proposition \ref{prop_low_conv_R}. Thus $\psi_R^\gamma$ satisfies
                \begin{equation*}
                    -\Delta \psi_R^\gamma + |D \psi_R^\gamma|^m \leq \gamma (f_R(x) - \nu_R) \leq \gamma(f(x) -  \nu_R),
                \end{equation*}
            where the second inequality follows from the definition of $f_R$ as an extension of $\min\{f,R\}$.
            
            Consider $\min_{\RN} (w - \psi_R^\gamma)$. This minimum is attained since $\psi_R$ is periodic and therefore bounded in $\RN$ (though the bound may depend on $R$), while $w$ is coercive by Lemma \ref{lemma:supersol_growth}. Thus we may use $ \psi_R^\gamma\in C^2$ as a test function for \eqref{erg_prob}. 
            
            Setting $\gamma_R = 1 + \nu_R - \lambda^*$, and noting the definition of $\lambda^*$ implies that $\gamma_R < 1$, we can argue as in Lemma \ref{lemma_simplicity} to show that $\min_{\RN} (w - \psi_R^{\gamma_R})$ is attained in a fixed compact set. Letting $R\to +\infty$ gives $w(x) \geq \phi(x) + c$ for some $c\in \R$. 
            
            To show equality, we argue exactly as in the second part of Lemma \ref{lemma_simplicity}, considering the difference $w - \phi^\gamma$ for $\gamma<1$, where $\phi$ is the unique solution of \eqref{erg_prob}.
        \end{proof}
        
       \begin{remark}\label{rmk_differences_of _approximations}
            Both $(\lambda_R, \phi_R)$ and $(\nu_R, \psi_R)$ are approximations of the unique bounded-from-below solution $(\lambda^*, \phi)$ given by Theorem \ref{thm_uniqueness} and Proposition \ref{prop_simplicity}. However, there are a couple of differences which allow us to obtain Lemma \ref{lemma_simplicity} and Corollary \ref{cor_lambdas} for the family $(\lambda_R, \phi_R)$ without appealing to the results of \cite{arapostathis2019uniqueness}, while the analogous results for $(\nu_R, \psi_R)$ cannot be obtained in this way, at least with similar techniques: first, since $\lambda_R$ is related to the ergodic problem with \textit{state-constraints} boundary condition, we have the extremal characterization \eqref{char_lambda_R}, which implies the monotonicity of $(\lambda_R)_R$, and in particular that $\lambda_R\geq \bar\lambda$; second, $\lambda^*$ is defined as a critical value for the existence of \textit{subsolutions} of \eqref{erg_prob}, and the argument of Corollary \ref{cor_lambdas} relies heavily on this fact.
        \end{remark}
        
    \section{Large-time behavior}\label{sec_large}
        
        \begin{proposition}\label{prop_unif_bound}
            For any compact $K\subset \RN$, $u(x,t) - \lambda^*t$ is bounded over $K\times(0,+\infty)$. Equivalently,
                \begin{equation*}
                    \frac{u(\cdot,t)}{t} \to \lambda^* \quad\text{locally uniformly in }\RN.
                \end{equation*}
        \end{proposition}
        
        \begin{proof}
            Let $K\subset\RN$. For $R>0$, We define $V_R(x, t) = \phi_R(x) +\lambda_R t + \sup_{B_R} (u_0-\phi_R)$, where $(\phi_R, \lambda_R)$ are defined as in \ref{thm_bar_lambda}. We note that $V_R$ and $u$ are respectively a supersolution and a subsolution of the \emph{parabolic} state-constraint problem in $B_R$,
                \begin{equation*}
                    \left\{ 
                        \begin{array}{cl}
                            v_t -\Delta v+ |Dv|^m = f(x) &  \textrm{in }B_R\\
                            v_t -\Delta v+ |Dv|^m \geq f(x) &  \textrm{on }\partial B_R.\\
                            v(\cdot, 0) = u_0 &  \textrm{in }\overline{B}_R.\\
                        \end{array}
                    \right.
                \end{equation*}
            Hence, by the comparison principle over a bounded domain (see \cite{barles2004generalized}), $V_R(x,t) \geq u(x,t)$ in $B_R \times (0,+\infty)$. It follows that
                \begin{align*}
                    & \frac{u(x,t)}{t} \leq \lambda_R + \frac{\phi_R(x) + \sup_{B_R} (u_0-\phi_R)}{t},\\
                    & \limsup_{t\to+\infty} \frac{u(x,t)}{t} \leq \lambda_R.
                \end{align*}
            By Theorem \ref{thm_bar_lambda} and Corollary \ref{cor_lambdas}, we have $\lambda_R \to \lambda^*$ as $R\to+\infty$. Thus, taking $R\to +\infty$ in the last inequality we have
                \begin{equation*}
                    \limsup_{t\to+\infty} \frac{u(x,t)}{t} \leq \lambda^*.
                \end{equation*}
            The lower limit for $\frac{u(x,t)}{t}$ is obtained similarly, by considering the solution pair $(\psi^R, \nu_R)$ of Proposition \ref{prop_low_conv_R}. In this case we compare the solution $u(x,t)$ to
                \begin{equation*}
                    U_R(x,t) :=  \psi_R(x) +\nu_R t + \inf_{B_R} (u_0-\psi_R)
                \end{equation*}
            over $\RN\times (t_0, +\infty)$ for some $t_0>0$. Despite lacking a general comparison principle for equation \eqref{vhj_whole}, this comparison is possible because we know $\psi_R$ is bounded (since it is periodic) and $u(x,t_0)\to+\infty$ as $|x|\to+\infty$ for any fixed  $t_0>0$, by Lemma \ref{lemma:supersol_growth}. The limit on the right hand side gives the desired lower bound by Theorem \ref{thm_uniqueness}, which in particular implies that $\nu_R\to \lambda^*$.
        \end{proof}
        
    \subsection{Proof of Theorem \ref{thm_main}}
        
        \begin{proof}[Proof of Theorem \ref{thm_main}]
        
         \noindent \textit{Step 1:} Let $v(x,t) = u(x,t) - \lambda^*t$ and define $\bar{v}(x,t) = \limsup\limits_{t\to+\infty}\ \!\!^* v(x,t)$. By stability, $\bar{v}$ is a subsolution of \eqref{erg_prob}$_{\lambda^*}$, and is bounded from below by Proposition \ref{prop_unif_bound}. Hence, by Lemma \ref{prop_simplicity}, there exists a constant $\hat{c}$ such that $\bar{v}(x) = \phi(x) + \hat{c}$ for all $x\in \RN$.\\
         
         \noindent \textit{Step 2:} We claim that there exists a sequence $(t_k)_{k\in \mathbb{N}}$ such that
             \begin{equation}\label{seq_locu_conv}
                 v(\cdot, t_k) \to \phi + \hat{c}\quad \textrm{uniformly over } K \textrm{ as } t_k \to +\infty,
             \end{equation}
         for any compact $K\subset \RN$. This follows from the same compactness argument given in the proof of Theorem 2 in \cite{barles2020large} for the case $m>2$, which we include here for convenience. 
         
         Fix $\hat{x}\in\RN$. By the definition of half-relaxed limits, there exists a sequence $(x_n, t_n)\in\RN\times(0,+\infty)$ such that $x_n\to \hat{x}, \ t_n\to\infty$, and $v(x_n,t_n)\to\bar{v}(\hat{x})$. Consider $v_n(\cdot):=v(\cdot, t_{n} - 1)$. By Proposition \ref{prop_unif_bound}, the sequence $(v_n)$ is uniformly bounded over compact sets. Furthermore, by the local gradient bound of Theorem \ref{gradBound} (in the Appendix), it is also uniformly equicontinuous over compact sets. Thus, there exists $w_0\in C(\RN)$ such that, given a compact $K\subset \RN$, there exists a subsequence $(v_{n'})$ such that $v_{n'}\to w_0$ uniformly over $K$ as $n'\to\infty$.
         
         Consider now
             \begin{equation*}
             w_{n'}(x,t) = v(x, t + t_{n'} - 1)\quad\textrm{for }(x,t)\in\RN\times(0,+\infty).
             \end{equation*}
         The sequence $(w_{n'})$ is again uniformly equicontinuous (in both space and time variables) due to Corollary \ref{gradBoundCor}. Thus, given also $T>0$, we have that 
             \begin{equation}\label{defW}
             w_{n'}\to w \textrm{ uniformly over } K\times[0,T] \quad\textrm{ for some }w \in C(\RN\times (0,+\infty))
             \end{equation}
         (again passing to a subsequence if necessary---this is omitted for ease of notation). By the definition of the half-relaxed limit $\bar{v}$, $w(x,t)\leq\bar{v}(x)=\phi(x) + \hat{c}$ for all $(x,t)\in\RN\times(0,+\infty)$, and by construction,
             \begin{equation}
             w(\hat{x},1) = \lim_{n'} v(\hat{x},t_{n'}) = \phi(\hat{x}) + \hat{c}.
             \end{equation}
         Hence, by the Strong Maximum Principle (see Lemma \ref{strongMaxP} in the Appendix), $w - (\phi + \hat{c})$ is constant in $\RN\times[0,1]$. In particular, $w_0(x)=\phi(x) + \hat{c}$ for all $x\in\RN$. Thus, given a compact $K\subset \RN$, there exists a subsequence $(v_{n'})$ such that $v_{n'}\to \phi(x) + \hat{c}$ uniformly over $K$ as $n'\to\infty$, and the claim follows from a simple diagonal argument.
       
       \noindent\textit{Step 3:} 
       Let $(\lambda^*,\phi)$ denote the solution pair of \eqref{erg_prob} such that $\min_{\RN} \phi = 0$; and let $(\lambda_R, \phi_R)$ denote the solution pair of \eqref{ergodic_SC_BR}, normalized so that $\phi_R\to \phi$ locally uniformly in $\RN$. 
       
       As in Lemma \ref{lemma_simplicity}, set $\mu_R = 1 + \lambda^*-\lambda_R > 1$, define $\phi_R^{\mu_R} = \mu_R\phi_R$, and denote by $x_R$ any point where $m_R:= \min_{B_R} (\phi_R^{\mu_R} - \phi)$ is achieved. We will show that $m_R= o(1)$ as $R\to+\infty$.
       
       Indeed, arguing as in Lemma \ref{lemma_simplicity}, we have that $x_R$ remains in a fixed compact set as $R\to+\infty$, say $\bar{B}_{\bar{R}}$. Thus, using the uniform convergence of $\phi_R^{\mu_R}\to \phi$ over $\bar{B}_{\bar{R}}$, we have
               \begin{align*}
                   m_R= \min_{B_R} (\phi_R^{\mu_R} - \phi) = \min_{B_{\bar{R}}} (\phi_R^{\mu_R} - \phi) \to 0, \quad \text{as } R\to+\infty.
               \end{align*}
       Let $\hat{K}\subset\RN$, $\epsilon>0$, and $R>0$ such that $\hat{K}\subset B_R$. Using the conslusion of \textit{Step 2} above, let $n_0\in \mathbb{N}$ such that for all $n\geq n_0$,
               \begin{equation*}
                   v(x, t_n) \leq \phi(x) + \hat{c} + \epsilon, \quad \textrm{for all } x\in \bar{B}_R.
               \end{equation*}
           
           We will construct a supersolution of the \emph{parabolic} state-constraints problem in $B_R$ to bound $v(x, t + t_n)$ by above for $x\in \hat{K}$ and $t>0$. Write $M_R = (-m_R)^+$, and define
               \begin{equation*}
                   V_R (x,t) = \phi_R^{\mu_R} + \hat{c} + \epsilon + (\mu_R\lambda_R - \lambda^*)t + M_R, \quad x\in B_R, \ t>0.
               \end{equation*}
          
          Recalling $\phi_R^{\mu_R}$ satisfies
              \begin{equation}\label{eq_mu_R}
                  -\Delta \phi_R^{\mu_R} + {\mu_R}^{1-m}|D\phi_R^{\mu_R}|^m = {\mu_R}(f(x) - \lambda_R)\quad\textrm{in } B_R,
              \end{equation}      
          we have 
               \begin{align*}
                   \partial_t V_R - \Delta V_R + |DV_R|^m ={}& {\mu_R}(f(x) - \lambda_R) + {\mu_R}\lambda_R - \lambda^* \\
                   \geq{}& f(x) - \lambda^* \qquad \textrm{in }B_R\times (0,+\infty),
               \end{align*}
           where we have used that $\mu_R>1$ and the assumption that $f\geq 0$. Furthermore, for any $t\geq 0$, 
               \begin{equation*}
                   V_R(x,t)\to +\infty \quad\textrm{as }x\to \partial B_R,
               \end{equation*}
            since $V_R(x,t)\geq \phi_R^{\mu_R}(x)$ and $\phi_R^{\mu_R}$ has this property. Finally, 
               \begin{align*}
                   V_R(x,0) &{}= \phi_R^{\mu_R}(x) + M_R + \hat{c} + \epsilon\\
                   &{}\geq \phi_R^{\mu_R}(x) + \max_{B_R}(\phi - \phi_R^{\mu_R}) + \hat{c} + \epsilon\\
                   &{}\geq \phi(x) + \hat{c} + \epsilon, \quad\textrm{for all } x\in B_R.
               \end{align*}
           
           Thus, by comparison (over a bounded domain) we have
               \begin{equation*}
                   V_R (x,t) \geq v(x, t + t_n) \quad \text{in } B_R \times (0, +\infty).
               \end{equation*}
           In particular, by the initial choice of $R>0$, this holds in $\hat{K}\times(0,+\infty)$.
           
           Fix $\hat{t}>0$. The previous inequality establishes that
                   \begin{equation*}
                       v(x, \hat{t} + t_n) \leq \phi_R^{\mu_R}(x) + ({\mu_R}\lambda_R - \lambda^*)\hat{t} + M_R + \hat{c} + \epsilon, \quad \text{for all } x\in \hat{K}.
                   \end{equation*}
               Sending $R\to +\infty$, by the uniform convergence of $\phi_R^{\mu_R}\to \phi$ over $\hat{K}$, using also that $M_R\to 0$ and $\mu_R\to 1$, we have
                   \begin{equation}\label{final_up_bound}
                       v(x, \hat{t} + t_n) \leq \phi(x) + \hat{c} + \epsilon, \quad\text{for all } x\in \hat{K}.
                   \end{equation}
               Since $\hat{K}$ and the right-hand side of the previous inequality is independent of $\hat{t}$, this provides the desired upper bound.\\
               
           \noindent\textit{Step 4.} To obtain the lower bound corresponding to \eqref{final_up_bound}, we define
                \begin{equation*}
                    \tilde{U}_R(x,t) = \psi_R^{\gamma_R}(x) + \hat{c} - \epsilon + (\gamma_R\nu_R - \lambda^*)t + \tilde{m}_R,
                \end{equation*}
           where $\tilde{m}_R = \min_{\RN} (\phi - \psi_R^{\gamma_R})$. Arguing as before, we can show that this minimum is achieved in a fixed compact set and therefore $\tilde{m}_R\to 0$ as $R\to+\infty$.
           
           Furthermore,
                \begin{equation*}
                    \tilde{U}_R(x,0) = \psi_R^{\gamma_R}(x) + \hat{c} - \epsilon + \tilde{m}_R \leq \phi(x) + \hat{c} - \epsilon,
                \end{equation*}
          and
            \begin{align*}
                &\partial_t\tilde{U}_R - \Delta \tilde{U}_R + |D\tilde{U}_R|^m \leq \gamma_R(f_R(x) - \nu_R) + \gamma_R\nu_R - \lambda^*\\
                &\qquad = f_R(x) - \lambda^* \leq f(x) - \lambda^*.
            \end{align*}
           Therefore, $\tilde{U}_R(x,t) \leq v(x, t + t_n)$ for all $x\in \RN, \ t>0$. In particular, this holds over $\hat{K}\times (0,+\infty)$. Hence, for fixed $\hat{t}>0$, taking $R\to+\infty$ in $\tilde{U}_R(x,t) \leq v(x, \hat{t} + t_n)$ we arrive at
                \begin{equation*}
                    \phi(x) + \hat{c} - \epsilon \leq v(x,\hat{t} + t_n)\quad \text{for all }x\in \hat{K},
                \end{equation*}
           and with this we conclude.
        \end{proof}
        
      \appendix
      \section{} 
        
        \begin{theorem}[Local Gradient Bounds]\label{gradBound}
        Let $R, \tau>0$.
        	\begin{enumerate}[label=(\textit{\alph*})]
        	\item\label{gradBerg} There exists $K_2>0$ depending only on $m$ and $N$, such that for any $R \geq R' + 1 > 0$ the solution of \eqref{erg_prob} satisfies
        		\begin{equation*}
        			\sup_{B_{R'}} |D\phi| \leq K_2(1 + \sup_{B_R} |f|^{\frac{1}{m}} + \sup_{B_R} |Df|^{\frac{1}{2m-1}}).
        		\end{equation*}
        	
        	\item\label{gradBparab} If $u$ is a solution of
        		\begin{align}
        			u_t - \Delta u + |Du|^m = f(x) &{}\quad\textrm{in } \Omega \times (0,T], \label{HJO}\\
        			u(x, 0) = u_0(x) &{}\quad\textrm{in } \overline{\Omega},\label{HJO-i}
        		\end{align}
        	where $\Omega$ is a domain of $\R^N$ such that $B_{R+1}\subset\Omega$ and $f\in \locLip(\RN)$, then $u$ is Lipschitz continuous in $x$ in $B_R\times[\tau, +\infty)$ and $|Du(x,t)|\leq L$ for a.e. $x\in B_R$, for all $t\geq\tau$, where $L$ depends on $R$ and $\tau$. Moreover this result holds with $\tau=0$ if $u_0$ is locally Lipschitz continuous in $\Omega$.
        	\end{enumerate}
        \end{theorem}
        
        Both results in Theorem \ref{gradBound} are classical. The estimate in \ref{gradBerg} appears as stated in \cite{ichihara2012large}, but can also be inferred from the results of \cite{lasry1989nonlinear} (see also \cite{lions1980resolution}, \cite{lions1985quelques}).
        
        The conclusion of \ref{gradBparab} can also be adapted from the results of \cite{lions1982generalized} (that recovers some of the results from \cite{lions1980resolution}). For a proof closer to our setting---namely, within the context of viscosity solutions, via the \emph{weak 
        Bernstein method}---we refer the reader to Theorem 4.1 in \cite{barles2021local}. The viscous Hamilton-Jacobi Equation (\ref{HJO}) is easily shown to satisfy the structure conditions required therein. 
        
        	\begin{corollary}\label{gradBoundCor}
        		Let $R, \tau>0$. The solution $u$ of \eqref{HJO}-\eqref{HJO-i} is H\"older-continuous of order $1/2$ in $B_R\times[\tau, +\infty)$ and $\|u\|_{C^{0,1/2}(B_R\times[\tau, +\infty))}\leq M$ for some $M>0$ depending on $R, \tau$, the data $u_0, f$ and universal constants. Moreover, this result holds with $\tau=0$ if $u_0$ is locally Lipschitz continuous in $\Omega$.
        	\end{corollary}
        	
        	\begin{proof} The proof of Corollary~\ref{gradBoundCor} can be done in two ways: either by using classical interior parabolic estimates (see \cite{wang1992regularity}, Theorem 4.19, and also Theorem 4.36 in \cite{imbert2013introduction}), in which case the H\"older-regularity of the solution $u$ is of some order $\alpha\in (0,1)$ depending on universal constants, or by the argument of \cite{barles2003uniqueness}, which implies that a solution which is Lipschitz in $x$ is $1/2$-H\"older-continuous in $t$.
        	
        	In both proofs, the result relies on part \ref{gradBparab} of Theorem \ref{gradBound}, which implies that the solution $u$ of \eqref{vhj_whole} satisfies $|Du|\leq L$ in $B_{R}\times (\tau, +\infty)$, and therefore that $u_t - \Delta u $ is bounded in $B_{R}\times (\tau, +\infty)$. This allows the use of either of the two arguments just mentioned.
        	\end{proof} 
            
            \begin{lemma}[Strong Maximum Principle]\label{strongMaxP}
            	Let $R,C>0$. Any upper semicontinuous subsolution of
            		\begin{equation}\label{strongMaxP_eq}
            			u_t - \Delta u - C|Du| = 0 \quad\text{in } B_R \times (0, +\infty)
            		\end{equation}
            	that attains its maximum at some $(x_0, t_0)\in B_R \times (0, +\infty)$ is constant in $B_R \times [0, t_0]$.
            \end{lemma}
            
            \begin{proof}
                See, e.g., \cite{da2004remarks}, Corollary 2.4 and \cite{bardi1999strong}, Corollary 1. (The latter result concerns time-independent equations, but the method of proof equally applies in this context.)
            \end{proof}
            
    \noindent\textbf{Acknowledgments.} A.Q.~ was partially supported by FONDECYT Grant \# 1190282 and Programa Basal, CMM. U. de Chile. A.R.~was partially supported by Fondecyt, Postdoctorado Nacional 2019, Proyecto Nº 3190858.

\vspace{5mm}

\noindent \textsc{Alexander Quaas}\\
\noindent \textit{Email:} alexander.quaas@usm.cl\\
\noindent \textsc{Departamento de Matemática, Universidad Técnica Federico Santa María, Casilla V-110, Avda.~España 1680, Valparaíso, Chile.}\\[4pt]

\noindent \textsc{Andrei Rodríguez-Paredes}\\
\noindent \textit{Email:} andrei.rodriguez@usach.cl\\
\noindent \textsc{Departamento de Matemática y Ciencia de la Computación, Universidad de Santiago de Chile, Avda.~Libertador General Bernardo O'Higgins 3383, Santiago, Chile.}
        
 \end{document}